\documentclass[12pt, reqno]{amsart}
\usepackage{amsmath, amsthm, amscd, amsfonts, amssymb, graphicx, color}
\usepackage[bookmarksnumbered, colorlinks, plainpages]{hyperref}
\hypersetup{colorlinks=true,linkcolor=red, anchorcolor=green, citecolor=cyan, urlcolor=red, filecolor=magenta, pdftoolbar=true}

\newtheorem{theorem}{Theorem}[section]
\newtheorem{lemma}[theorem]{Lemma}

\newtheorem{corollary}[theorem]{Corollary}
\theoremstyle{definition}
\newtheorem{definition}[theorem]{Definition}

\theoremstyle{remark}

\numberwithin{equation}{section}

\begin{document}

\setcounter{page}{1}

\title[$A$-Birkhoff--James orthogonality of operators]
{Birkhoff--James orthogonality of operators in semi-Hilbertian spaces and its applications}

\author[A. Zamani]{Ali Zamani}

\address{Department of Mathematics, Farhangian University, Tehran, Iran.}
\email{\textcolor[rgb]{0.00,0.00,0.84}{zamani.ali85@yahoo.com}}

%\dedicatory{This paper is dedicated to Professor ABCD}

\let\thefootnote\relax\footnote{Copyright 2018 by the Tusi Mathematical Research Group.}

\subjclass[2010]{Primary 46C05; Secondary 47B65, 47L05.}

\keywords{Positive operator, semi-inner product,
$A$-Birkhoff--James orthogonality, and $A$-distance formulas.}

\date{Received: xxxxxx; Revised: yyyyyy; Accepted: zzzzzz.}

\begin{abstract}
In this paper, the concept of Birkhoff--James orthogonality of operators
on a Hilbert space is generalized when a semi-inner product is considered.
More precisely, for linear operators $T$ and $S$ on a complex Hilbert space $\mathcal{H}$,
a new relation  $T\perp^B_A S$ is defined if $T$ and $S$
are bounded with respect to the seminorm induced by a positive
operator $A$ satisfying ${\|T + \gamma S\|}_A\geq {\|T\|}_A$ for all $\gamma \in \mathbb{C}$.
We extend a theorem due to R. Bhatia and P. \v{S}emrl, by proving that $T\perp^B_A S$
if and only if there exists a sequence of $A$-unit vectors $\{x_n\}$ in $\mathcal{H}$ such that
$\displaystyle{\lim_{n\rightarrow +\infty}}{\|Tx_n\|}_A = {\|T\|}_A$
and $\displaystyle{\lim_{n\rightarrow +\infty}}{\langle Tx_n, Sx_n\rangle}_A = 0$.
In addition, we give some $A$-distance formulas. Particularly, we prove
\begin{align*}
\displaystyle{\inf_{\gamma \in \mathbb{C}}}{\|T + \gamma S\|}_{A} =
\sup\Big\{|{\langle Tx, y\rangle}_A|; \,  {\|x\|}_{A} = {\|y\|}_{A} = 1, \, {\langle Sx, y\rangle}_A = 0\Big\}.
\end{align*}
Some other related results are also discussed.
\end{abstract} \maketitle

\section{\textbf{Introduction and preliminaries}}
Let $\mathbb{B}(\mathcal{H})$ denote the $C^{\ast}$-algebra of all bounded
linear operators on a complex Hilbert space $\mathcal{H}$ with an inner
product $\langle \cdot,\cdot \rangle$ and the corresponding norm
$\|\cdot\| $. The symbol $I$ stands for the identity operator
on $\mathcal{H}$.
If $T\in\mathbb{B}(\mathcal{H})$, then we
denote by $\mathcal{R}(T)$ and $\mathcal{N}(T)$ the range and
the kernel of $T$, respectively, and by $\overline{\mathcal{R}(T)}$
the norm closure of $\mathcal{R}(T)$.
Throughout this article, we assume that $A\in\mathbb{B}(\mathcal{H})$ is a positive operator
and that $P$ is the orthogonal projection onto $\overline{\mathcal{R}(A)}$.
Recall that $A$ is called positive if $\langle Ax, x\rangle\geq 0$ for all $x\in\mathcal{H}$.
Such an $A$ induces a positive semidefinite sesquilinear
form ${\langle \cdot, \cdot\rangle}_A: \,\mathcal{H}\times \mathcal{H} \rightarrow \mathbb{C}$
defined by
\begin{align*}
{\langle x, y\rangle}_A = \langle Ax, y\rangle, \qquad x, y\in\mathcal{H}.
\end{align*}
Denote by ${\|\cdot\|}_A$ the seminorm induced by ${\langle \cdot, \cdot\rangle}_A$, that is,
${\|x\|}_A = \sqrt{{\langle x, x\rangle}_A}$
for every $x\in\mathcal{H}$.
It can be easily seen that ${\|\cdot\|}_A$ is a norm if and only if $A$ is an injective operator,
and  that $(\mathcal{H}, {\|\cdot\|}_A)$ is a complete space if and only if $\mathcal{R}(A)$ is closed in $\mathcal{H}$.
For $x, y\in\mathcal{H}$, we say that $x$ and $y$ are $A$-orthogonal, denoted by $x \perp_A y$,
if ${\langle x, y\rangle}_A = 0$. Note that this definition is a natural extension of
the usual notion of orthogonality, which represents the $I$-orthogonality case.
Furthermore, we put
\begin{align*}
\mathbb{B}_{A^{1/2}}(\mathcal{H}) =\Big\{T\in \mathbb{B}(\mathcal{H}):
\, \exists \, c> 0\, \forall x\in\mathcal{H}; {\|Tx\|}_A \leq c{\|x\|}_A\Big\}.
\end{align*}
We say that an operator $T\in \mathbb{B}(\mathcal{H})$ is $A$-bounded
if $T$ belongs to $\mathbb{B}_{A^{1/2}}(\mathcal{H})$.
It can be shown that $\mathbb{B}_{A^{1/2}}(\mathcal{H})$ is a unital
subalgebra of $\mathbb{B}(\mathcal{H})$ which, in general,
is neither closed nor dense in $\mathbb{B}(\mathcal{H})$ (see \cite{Ar.Co.Go.1}).
We equip $\mathbb{B}_{A^{1/2}}(\mathcal{H})$ with
the seminorm ${\|\cdot\|}_A$ defined as follows:
\begin{align*}
{\|T\|}_A = \displaystyle{\sup_{x \in \overline{\mathcal{R}(A)}, x\neq 0}} \frac{{\|Tx\|}_A}{{\|x\|}_A}
= \inf\Big\{c>0; \, {\|Tx\|}_A \leq c{\|x\|}_A, x\in \mathcal{H}\Big\}< \infty.
\end{align*}
In addition, for $T\in \mathbb{B}_{A^{1/2}}(\mathcal{H})$, we have
\begin{align*}
{\|T\|}_A = \displaystyle{\sup_{x\in \mathcal{H}, {\|x\|}_A = 1}} {\|Tx\|}_A
= \sup\Big\{|{\langle Tx, y\rangle}_A|; \, x, y\in \mathcal{H}, {\|x\|}_A = {\|y\|}_A = 1\Big\}.
\end{align*}
%%%%%%%%%%%%%%%%%
Of course, many difficulties arise. For instance, it may happen
that ${\|T\|}_A = \infty$ for some $T\in \mathbb{B}(\mathcal{H})$.
In addition, not any operator admits an adjoint operator for the
semi-inner product ${\langle \cdot, \cdot\rangle}_A$.
For more details about this class of operators, we refer the reader to \cite{Ar.Co.Go.1}.
In recent years, several results covering some classes of operators
on a complex Hilbert space $\big(\mathcal{H}, \langle \cdot, \cdot\rangle\big)$
are extended to $\big(\mathcal{H}, {\langle \cdot, \cdot\rangle}_A\big)$; see \cite{Ar.Co.Go.1, Ba.Ka.Ah} and their references.
%%%%%%%%%%%%%%%%%%%%

The notion of orthogonality in $\mathbb{B}(\mathcal{H})$ can be introduced
in many ways ( see, e.g., \cite{M.Z}).
When $T, S\in \mathbb{B}(\mathcal{H})$, we say that $T$ is
Birkhoff--James orthogonal to $S$, denoted $T\perp^B S$, if
\begin{align*}
\|T + \gamma S\|\geq \|T\|\quad \mbox{for all} \,\, \gamma \in \mathbb{C}.
\end{align*}
%%%%%%%%%%%%%%%%%%%%%
In Hilbert spaces, this orthogonality
is equivalent to the usual notion of orthogonality.
This notion of orthogonality plays
a very important role in the geometry of Hilbert space operators.
%%%%%%%%%%%%%%%%%%%%%%%%%%%%%%%%%%%%%%%%%%%
%%%%%%%%%%%%%%%%%%%%%%%%%%%%%%%%%%%%%%%%%
For $T, S\in \mathbb{B}(\mathcal{H})$, Bhatia and \v{S}emrl in \cite[Remark 3.1]{B.S}
and Paul in \cite[Lemma 2]{Pa} independently proved that $T\perp^B S$ if and only if there exists
a sequence of unit vectors $\{x_n\}$ in $\mathcal{H}$ such that
\begin{align*}
\lim_{n\rightarrow\infty} \|Tx_n\| = \|T\| \quad
\mbox{and} \quad \lim_{n\rightarrow\infty}\langle Tx_n, Sx_n\rangle = 0.
\end{align*}
It follows then that if the Hilbert space $\mathcal{H}$ is finite-dimensional,
$T\perp^B S$ if and only if there is a unit vector $x\in\mathcal{H}$ such that
$\|Tx\| = \|T\|$ and $\langle Tx, Sx\rangle = 0$.

%%%%%%%%%%%%%%%%%%%%%%%%%%%%%
Recently, some authors extended the well known result of Bhatia--\v{S}emrl (see \cite{Ch.St.Wo, Wo.3, Wo.1}).
Moreover, the papers \cite{Wo.3} and \cite{Wo.1} show another ways to obtain the Bhatia--\v{S}emrl theorem.
Some other authors studied different aspects of orthogonality of
operators on various Banach spaces and elements of an arbitrary Hilbert $C^*$-module;
 see, for instance,\cite{A.R, B.C.M.W.Z, Ch.Wo, G.S.P, Ke, P.S.M.M, Wo.2, M.Z.2}.
%%%%%%%%%%%%%%%%%%%%%%%%%%%5

Now, let us introduce the notion of $A$-Birkhoff--James orthogonality of operators in semi-Hilbertian spaces.
\begin{definition}\label{de.31}
An element $T \in \mathbb{B}_{A^{1/2}}(\mathcal{H})$ is called an $A$-Birkhoff--James orthogonal
to another element $S \in \mathbb{B}_{A^{1/2}}(\mathcal{H})$, denoted by $T\perp^B_A S$, if
\begin{align*}
{\|T + \gamma S\|}_A\geq {\|T\|}_A \quad \mbox{for all} \,\, \gamma \in \mathbb{C}.
\end{align*}
\end{definition}
It is a generalization of the notion of Birkhoff--James of Hilbert space operators.
Notice that the $A$-Birkhoff--James orthogonality is homogenous,
that is, $T\perp^B_A S\, \Leftrightarrow \, (\alpha T)\perp^B_A (\beta S)$ for all $\alpha, \beta\in\mathbb{C}$.

The paper is organized as follows:
In the next section, we obtain some characterizations of $A$-Birkhoff--James orthogonality for bounded
linear operators in semi-Hilbertian spaces.
In particular, for $T, S \in \mathbb{B}_{A^{1/2}}(\mathcal{H})$, we show that $T\perp^B_A S$
if and only if there exists a sequence of $A$-unit vectors $\{x_n\}$ in $\mathcal{H}$ such that
\begin{align*}
\displaystyle{\lim_{n\rightarrow +\infty}}{\|Tx_n\|}_A = {\|T\|}_A
\quad \mbox{and} \quad \displaystyle{\lim_{n\rightarrow +\infty}}{\langle Tx_n, Sx_n\rangle}_A = 0.
\end{align*}
Furthermore, for the finite-dimensional Hilbert space $\mathcal{H}$, we show that $T\perp^B_A S$
if and only if there exists an $A$-unit vector $x\in \mathcal{H}$ such that
${\|Tx\|}_A = {\|T\|}_A$ and ${\langle Tx, Sx\rangle}_A = 0$.
The mentioned property extends the Bhatia--\v{S}emrl theorem.

In the last section, some formulas for the $A$-distance of an operator to the
class of multiple scalars of another one in semi-Hilbertian spaces are given.
In particular, we show that
\begin{align*}
\displaystyle{\inf_{\gamma \in \mathbb{C}}}{\|T + \gamma S\|}_{A} =
\sup\Big\{|{\langle Tx, y\rangle}_A|; \,  {\|x\|}_{A} = {\|y\|}_{A} = 1, \, Sx \perp_A y\Big\}.
\end{align*}
We then apply it to prove that
$\displaystyle{\inf_{\gamma \in \mathbb{C}}}{\|T + \gamma S\|}^2_{A}=
\displaystyle{\sup_{{\|x\|}_A = 1}}\Phi^{(T, S)}_A(x)$,
where
\begin{align*}
\Phi^{(T, S)}_A(x)=\begin{cases}
{\|Tx\|}_A^2 - \frac{|{\langle Tx, Sx\rangle}_A|^2}{{\|Sx\|}_A^2} &\text{if\, ${\|Sx\|}_A\neq0$,}\\
{\|Tx\|}_A^2       &\text{if\, ${\|Sx\|}_A=0$}.
\end{cases}
\end{align*}
Our results cover and extend the works of Fujii and Nakamoto in \cite{Fu.Na} and Bhatia and \v{S}emrl in \cite{B.S}.
%%%%%%%%%%%%%%%%%%%%%%%%%%%%%%%%%%
\section{$A$-Birkhoff--James orthogonality of operators}
We first prove a technical lemma that we need in what follows.
We use some techniques of \cite[Theorem 3.2]{Ba.Ka.Ah} to prove this result.
In fact, the following lemma extends Magajna's theorem \cite{Mag}.
%%%%%%%%%%%%%%%%%%%%%%%%%%%%%%%%%%%%%%%%%%%%
\begin{lemma}\label{L.2.1}
Let $T, S\in\mathbb{B}_{A^{1/2}}(\mathcal{H})$. Then the set
\begin{align*}
W_A(T, S) = \Big\{\xi\in \mathbb{C}; \, \exists \,\{x_n\}\subset \mathcal{H}, \,{\|x_n\|}_A = 1,
\, &\displaystyle{\lim_{n\rightarrow +\infty}}{\|Tx_n\|}_A = {\|T\|}_A,
\\&\qquad \qquad \mbox{and}\, \displaystyle{\lim_{n\rightarrow +\infty}}{\langle Tx_n, Sx_n\rangle}_A = \xi\Big\}
\end{align*}
is nonempty, compact, and convex.
\end{lemma}
\begin{proof}
Since the seminorm of $T\in\mathbb{B}_{A^{1/2}}(\mathcal{H})$ is given by
\begin{align*}
{\|T\|}_A = \sup\{{\|Tx\|}_A; \,x\in \overline{\mathcal{R}(A)}, \, {\|x\|}_A = 1\},
\end{align*}
there exists a sequence of $A$-unit vectors $\{x_n\}$ in $\overline{\mathcal{R}(A)}$ such that
\begin{align*}
\displaystyle{\lim_{n\rightarrow +\infty}}{\|Tx_n\|}_A = {\|T\|}_A.
\end{align*}
Furthermore, using the Cauchy--Schwarz inequality, we have
\begin{align*}
|{\langle Tx_n, Sx_n\rangle}_A| \leq {\|Tx_n\|}_A{\|Sx_n\|}_A \leq {\|T\|}_A{\|S\|}_A.
\end{align*}
Hence, $\{{\langle Tx_n, Sx_n\rangle}_A\}$ is a bounded sequence of complex
numbers, so there exists a subsequence $\{{\langle Tx_{n_k}, Sx_{n_k}\rangle}_A\}$
that converges to some $\xi_0 \in \mathbb{C}$. Thus $\xi_0 \in W_A(T, S)$ and hence $W_A(T, S)$ is nonempty.

On the other hand, considering the definition of $W_A(T, S)$ follows that
\begin{align*}
W_A(T, S) \subset \big\{\xi \in \mathbb{C}; \,|\xi|\leq {\|T\|}_A{\|S\|}_A\big\}.
\end{align*}
Therefore, to prove that $W_A(T, S)$ is compact, it is enough to show that $W_A(T, S)$ is closed.
Let $\xi_n \in W_A(T, S)$ and let $\displaystyle{\lim_{n\rightarrow +\infty}}\xi_n = \xi$.
Since $\xi_n \in W_A(T, S)$,
there exists a sequence of $A$-unit vectors $\{x^n_m\}$ in $\mathcal{H}$ such that
$\displaystyle{\lim_{m\rightarrow +\infty}}{\|Tx^n_m\|}_A = {\|T\|}_A$
and
$\displaystyle{\lim_{m\rightarrow +\infty}}{\langle Tx^n_m, Sx^n_m\rangle}_A = \xi_n$.
Now, let $\varepsilon>0$. Hence
\begin{align}\label{i.2.2.1}
\big|{\|Tx^n_m\|}_A - {\|T\|}_A\big| < \varepsilon
\end{align}
and also
\begin{align}\label{i.2.2.2}
\big|{\langle Tx^n_m, Sx^n_m\rangle}_A - \xi_n\big| < \frac{\varepsilon}{2}
\end{align}
for all sufficiently large $m$.
From (\ref{i.2.2.1}) and (\ref{i.2.2.2}), we get
\begin{align*}
\big|{\|Tx^n_m\|}_A - {\|T\|}_A\big| < \varepsilon
\end{align*}
and
\begin{align*}
\big|{\langle Tx^n_m, Sx^n_m\rangle}_A - \xi\big|
\leq \big|{\langle Tx^n_m, Sx^n_m\rangle}_A - \xi_n\big| + |\xi_n - \xi|
< \frac{\varepsilon}{2} + \frac{\varepsilon}{2} = \varepsilon
\end{align*}
for all sufficiently large $m$.
Therefore we deduce that
$\displaystyle{\lim_{m\rightarrow +\infty}}{\|Tx^n_m\|}_A = {\|T\|}_A$ and
$\displaystyle{\lim_{m\rightarrow +\infty}}{\langle Tx^n_m, Sx^n_m\rangle}_A = \xi$.
Thus $\xi \in W_A(T, S)$ and so $W_A(T, S)$ is closed.
%%%%%%%%%%%%%%%%%%%%%%%%%%%%%%%%%%%%%%%%%%%%%%%%

We next show that $W_A(T, S)$ is convex. Since $\mathcal{H}$ can be decomposed as
$\mathcal{H} = \mathcal{N}(A)\oplus \overline{\mathcal{R}(A)}$,
so every $x\in \mathcal{H}$ can be written in a unique way into
$x = y + z$ with $y\in \mathcal{N}(A)$ and $z\in \overline{\mathcal{R}(A)}$.
Furthermore, since $A \geq 0$, it follows that $\mathcal{N}(A)= \mathcal{N}(A^{1/2})$
which implies that ${\|x\|}_A = {\|z\|}_A$. Thus
\begin{align*}
W_A(T, S) &= \Big\{\xi\in \mathbb{C}; \, \exists \,\{(y_n, z_n)\}\subset \mathcal{N}(A)\times \overline{\mathcal{R}(A)},
\,{\|z_n\|}_A = 1, \,
\\&\displaystyle{\lim_{n\rightarrow +\infty}}{\|T(y_n + z_n)\|}_A = {\|T\|}_A, \,\mbox{and}\,
\displaystyle{\lim_{n\rightarrow +\infty}}{\langle Ty_n, Sz_n\rangle}_A + {\langle Tz_n, Sz_n\rangle}_A= \xi\Big\}.
\end{align*}
Since $T, S\in\mathbb{B}_{A^{1/2}}(\mathcal{H})$, then $T(\mathcal{N}(A))\subset \mathcal{N}(A)$
and $S(\mathcal{N}(A))\subset \mathcal{N}(A)$. Hence, we get
\begin{align*}
W_A(T, S) &= \Big\{\xi\in \mathbb{C}; \, \exists \,\{z_n\}\subset \overline{\mathcal{R}(A)},
\,{\|z_n\|}_A = 1, \,
\\& \qquad \qquad \displaystyle{\lim_{n\rightarrow +\infty}}{\|Tz_n\|}_A = {\|T\|}_A, \,\mbox{and}\,
\displaystyle{\lim_{n\rightarrow +\infty}}{\langle Tz_n, Sz_n\rangle}_A= \xi\Big\}
\\& = \Big\{\xi\in \mathbb{C}; \, \exists \,\{z_n\}\subset \overline{\mathcal{R}(A)},
\,{\|z_n\|}_A = 1, \,
\\& \qquad \qquad \displaystyle{\lim_{n\rightarrow +\infty}}{\|PTz_n\|}_A = {\|PT\mid_{\overline{\mathcal{R}(A)}}\|}_A,
\,\mbox{and}\, \displaystyle{\lim_{n\rightarrow +\infty}}{\langle PTz_n, PSz_n\rangle}_A= \xi\Big\}
\\& = W_{A_0}(\widetilde{T}, \widetilde{S}),
\end{align*}
where $A_0 = A\mid_{\overline{\mathcal{R}(A)}}$, $\widetilde{T} = PT\mid_{\overline{\mathcal{R}(A)}}$
and $\widetilde{S} = PS\mid_{\overline{\mathcal{R}(A)}}$.
By \cite[Lemma 2.1]{Mag}, we conclude that $W_A(T, S)$ is convex.
\end{proof}
%%%%%%%%%%%%%%%%%%%%%%%%%%%%%%%%%%
%%%%%%%%%%%%%%%%%%%%%%%%%%%%%%%%%%%%
Recall that the minimum modulus of $S\in\mathbb{B}(\mathcal{H})$ is defined by
\begin{align*}
m(S) = \inf \Big\{\|Sx\|: \, x\in \mathcal{H},\|x\| =1\Big\}.
\end{align*}
This concept is useful in studying linear operators (see \cite{M.Z},
and further references therein).
The $A$-minimum modulus of $S\in\mathbb{B}_{A^{1/2}}(\mathcal{H})$ can be defined by
\begin{align*}
m_A(S) =\inf\Big\{{\|Sx\|}_A: \,x\in \mathcal{H}, {\|x\|}_A = 1\Big\}.
\end{align*}
%%%%%%%%%%%%%%%%%%%%%%%%%%%%%%%%%%%%%%
We are now in a position to establish the main result of this section.
To establish the following theorem, we use some ideas of \cite[Theorem 2]{St}.
%%%%%%%%%%%%%%%%%%%%%%%%%%%%%%%%%%%%
\begin{theorem}\label{T.2.2}
Let $T, S\in\mathbb{B}_{A^{1/2}}(\mathcal{H})$. Then the following conditions are equivalent:
\begin{itemize}
\item[(i)] There exists a sequence of $A$-unit vectors $\{x_n\}$ in $\mathcal{H}$ such that
\begin{align*}
\displaystyle{\lim_{n\rightarrow +\infty}}{\|Tx_n\|}_A = {\|T\|}_A
\quad \mbox{and} \quad \displaystyle{\lim_{n\rightarrow +\infty}}{\langle Tx_n, Sx_n\rangle}_A = 0.
\end{align*}
\item[(ii)] ${\|T + \gamma S\|}^2_{A} \geq {\|T\|}^2_{A} + |\gamma|^2 m^2_A(S)$ for all $\gamma \in \mathbb{C}$.
\item[(iii)] $T\perp^B_A S$.
\end{itemize}
\end{theorem}
\begin{proof}
(i)$\Rightarrow$(ii) Suppose that (i) holds.
We have
\begin{align*}
{\|T + \gamma S\|}^2_A &\geq {\|(T + \gamma S)x_n\|}^2_A
\\& = {\|Tx_n\|}^2_A + \overline{\gamma}{\langle Tx_n, Sx_n\rangle}_A
+ \gamma {\langle Sx_n, Tx_n\rangle}_A + |\gamma|^2{\|Sx_n\|}^2_A
\end{align*}
for all $\gamma\in\mathbb{C}$ and $n\in\mathbb{N}$. Thus
\begin{align*}
{\|T + \gamma S\|}^2_A \geq {\|T\|}^2_A + |\gamma|^2\lim_{n\rightarrow\infty}\sup {\|Sx_n\|}^2_A
\geq {\|T\|}^2_A + |\gamma|^2m^2_A(S)
\end{align*}
for all $\gamma \in \mathbb{C}$.

(ii)$\Rightarrow$(iii) This implication is trivial.

(iii)$\Rightarrow$(i)
If ${\|S\|}_A = 0$, then since $T$ is a seminorm,
there exists a sequence of $A$-unit vectors $\{x_n\}$ in $\mathcal{H}$ such that
$\displaystyle{\lim_{n\rightarrow +\infty}}{\|Tx_n\|}_A = {\|T\|}_A$.
So, the Cauchy--Schwarz inequality implies
\begin{align*}
|{\langle Tx_n, Sx_n\rangle}_A| \leq {\|Tx_n\|}_A{\|Sx_n\|}_A \leq {\|T\|}_A{\|S\|}_A = 0.
\end{align*}
Hence, $\displaystyle{\lim_{n\rightarrow +\infty}}{\langle Tx_n, Sx_n\rangle}_A = 0$.
Now, let ${\|S\|}_A \neq 0$.
It is enough to show that $0\in W_A(T, S)$,
where $W_A(T, S)$ is defined in Lemma \ref{L.2.1}.
let $0\notin W_A(T, S)$.  Lemma \ref{L.2.1} implies that  $W_A(T, S)$ is
a nonempty compact and convex subset of the complex plane $\mathbb{C}$; hence
because of the rotation, we may suppose that $W_A(T, S)$ is contained in the
right half-plane. Therefore there is a line that separates $0$ from $W_A(T, S)$.
In other words, there exists $\tau > 0$ such that $\mbox{Re}W_A(T, S) > \tau$.
Let
\begin{align*}
\mathcal{H}_\tau = \Big\{x \in \mathcal{H}; \,{\|x\|}_A = 1, \,\mbox{and} \,\,\mbox{Re}W_A(T, S) \leq\frac{\tau}{2}\Big\}
\end{align*}
and
\begin{align*}
\delta = \sup\Big\{{\|Tx\|}_A; \, x\in \mathcal{H}_\tau\Big\}.
\end{align*}
We first claim that $\delta < {\|T\|}_A$.
Suppose $\delta \geq {\|T\|}_A$. Hence $\delta = {\|T\|}_A$. Thus there exists
a sequence of vectors $\{x_n\}$ in $\mathcal{H}_\tau$ such that
$\displaystyle{\lim_{n\rightarrow +\infty}}{\|Tx_n\|}_A = {\|T\|}_A$.
As $x_n \in \mathcal{H}_\tau$ so ${\|x_n\|}_A = 1$ and $\mbox{Re}W_A(T, S) \leq\frac{\tau}{2}$.
Now the sequence $\{{\langle Tx_n, Sx_n\rangle}_A\}$ is bounded, and hence it has a convergent subsequence, without
loss of generality, we can assume that $\{{\langle Tx_n, Sx_n\rangle}_A\}$
is convergent. If we set $\xi = \displaystyle{\lim_{n\rightarrow +\infty}}{\langle Tx_n, Sx_n\rangle}_A$,
then $\mbox{Re}(\xi) \leq \frac{\tau}{2}$ and this
contradicts the fact that $\mbox{Re}W_A(T, S) > \frac{\tau}{2}$. Thus $\delta < {\|T\|}_A$.
Let $\gamma_0 = \max\{\frac{-\tau}{2{\|S\|}^2_A}, \frac{\delta - {\|T\|}_A}{2{\|S\|}_A}\}$. Then $\gamma_0 <0$.
We claim that ${\|T + \gamma_0 S\|}_A < {\|T\|}_A$. Let $x$ be an $A$-unit vector in $\mathcal{H}$.
If $x \in \mathcal{H}_\tau$, then
\begin{align*}
{\|(T + \gamma_0 S)x\|}_A &\leq {\|Tx\|}_A + |\gamma_0|{\|Sx\|}_A \leq \delta - \gamma_0{\|S\|}_A
\\&\leq \delta + \frac{{\|T\|}_A - \delta}{2{\|S\|}_A}\,{\|S\|}_A = \frac{\delta}{2} + \frac{{\|T\|}_A}{2}
\end{align*}
and so ${\|(T + \gamma_0 S)x\|}_A \leq \frac{\delta}{2} + \frac{{\|T\|}_A}{2}$.

If $x\notin \mathcal{H}_\tau$, then we can write $Tx = (r + it)Sx + y$ with $r, t\in\mathbb{R}$ and
$Sx \perp_A y$. Thus
\begin{align*}
2r{\|S\|}^2_A \geq  2r{\|Sx\|}^2_A = 2\mbox{Re}{\langle Tx, Sx\rangle}_A > \frac{\tau}{2} \geq - \gamma_0 {\|S\|}^2_A,
\end{align*}
and hence $2r + \gamma_0 > 0$.
Now, let us put
\begin{align*}
\theta : = \inf\big\{{\|Sx\|}^2_A; \, x\notin \mathcal{H}_\tau, {\|x\|}_A = 1\big\}.
\end{align*}
Since $\gamma_0^2 + 2r\gamma_0 < 0$, we obtain
\begin{align*}
{\|(T + \gamma_0 S)x\|}^2_A &= {\Big\langle \big((r + \gamma_0) + it\big)Sx + y, \big((r + \gamma_0) + it\big)Sx + y\Big\rangle}_A
\\&= \big((r + \gamma_0)^2 + t^2\big){\|Sx\|}^2_A + {\|y\|}^2_A
\\&= {\|Tx\|}^2_A + (\gamma_0^2 + 2r\gamma_0){\|Sx\|}^2_A
\\&\leq {\|Tx\|}^2_A + (\gamma_0^2 + 2r\gamma_0)\inf\big\{{\|Sx\|}^2_A; \, x\notin \mathcal{H}_\tau, {\|x\|}_A = 1\big\}
\\&\leq {\|T\|}^2_A + (\gamma_0^2 + 2r\gamma_0)\theta.
\end{align*}
Hence ${\|(T + \gamma_0 S)x\|}^2_A\leq {\|T\|}^2_A + (\gamma_0^2 + 2r\gamma_0)\theta$.
Thus in all cases
\begin{align*}
{\|(T + \gamma_0 S)x\|}^2_A \leq
\max\Big\{\big(\frac{\delta}{2} + \frac{{\|T\|}_A}{2}\big)^2,
{\|T\|}^2_A + (\gamma_0^2 + 2r\gamma_0)\theta\Big\},
\end{align*}
whence
\begin{align*}
{\|T + \gamma_0 S\|}^2_A \leq \max\Big\{\big(\frac{\delta}{2} + \frac{{\|T\|}_A}{2}\big)^2,
{\|T\|}^2_A + (\gamma_0^2 + 2r\gamma_0)\theta\Big\}.
\end{align*}
Since $\max\Big\{\big(\frac{\delta}{2} + \frac{{\|T\|}_A}{2}\big)^2,
{\|T\|}^2_A + (\gamma_0^2 + 2r\gamma_0)\theta\Big\} < {\|T\|}^2_A$, we obtain
${\|T + \gamma_0 S\|}_A < {\|T\|}_A$.
Therefore we deduce that $T\not\perp^B_AS$ which
contradicts our hypothesis and
the proof is completed.
\end{proof}
%%%%%%%%%%%%%%%%%%%%%%%%%%%%%%%%%%%%
The following corollary gives a direct application of Theorem \ref{T.2.2} for the case $A = I$.
%%%%%%%%%%%%%%%%%%%%%%%%%%%%%%
\begin{corollary}$($see \cite[Remark 3.1]{B.S} and \cite[Lemma 2]{Pa}$)$
Let $\mathcal{H}$ be a complex Hilbert space and let $T, S\in\mathbb{B}(\mathcal{H})$.
Then the following statements are equivalent:
\begin{itemize}
\item[(i)] $T\perp^B S$.
\item[(ii)] There exists a sequence of unit vectors $\{x_n\}$ in $\mathcal{H}$ such that
\begin{align*}
\displaystyle{\lim_{n\rightarrow +\infty}}\|Tx_n\| = \|T\|
\quad \mbox{and} \quad \displaystyle{\lim_{n\rightarrow +\infty}}\langle Tx_n, Sx_n\rangle = 0.
\end{align*}
\end{itemize}
\end{corollary}
%%%%%%%%%%%%%%%%%%%%%%%%%%%%%%%%%%%%
In what follows, for $T\in\mathbb{B}_{A^{1/2}}(\mathcal{H})$, we denote $\mathbb{M}^T_A$
the set of all $A$-unit vectors at which $T$ attains the seminorm ${\|\cdot\|}_A$, that is,
\begin{align*}
\mathbb{M}^T_A = \Big\{x\in \mathcal{H}: {\|x\|}_A = 1,\,{\|Tx\|}_A = {\|T\|}_A\Big\}.
\end{align*}
%%%%%%%%%%%%%%%%%%%%%%%%%%%%%%%%%%%%%%%%%%%%%
For more information on norm-attaining sets, see \cite{F.G.M.R}.
In the next theorem, we consider a  finite-dimensional Hilbert space and characterize the $A$-Birkhoff--James orthogonality
of operators in semi-Hilbertian spaces.
\begin{theorem}\label{T.2.3}
Let $\mathcal{H}$ be a finite-dimensional Hilbert space and let $T, S\in\mathbb{B}_{A^{1/2}}(\mathcal{H})$.
Then the following conditions are equivalent:
\begin{itemize}
\item[(i)] There exists $x\in \mathbb{M}^T_A$ such that $Tx\perp_A Sx$.
\item[(ii)] $T\perp^B_A S$.
\end{itemize}
\end{theorem}
\begin{proof}
(i)$\Rightarrow$(ii)
Suppose that (i) holds. Then there exists an $A$-unit vectors $x\in \mathcal{H}$ such that
${\|Tx\|}_A = {\|T\|}_A$ and $Tx\perp_A Sx$. Put $x_n = x$ for all $n \in \mathbb{N}$.
So, by the equivalence (i)$\Leftrightarrow$(iii) in Theorem \ref{T.2.2}, we deduce that $T\perp^B_A S$.

(ii)$\Rightarrow$(i)
First note that, by using the decomposition $\mathcal{H} = \mathcal{N}(A)\oplus\overline{\mathcal{R}(A)}$
and letting $A_0 = A\mid_{\overline{\mathcal{R}(A)}}$, it can be seen that
the set $\{x\in \overline{\mathcal{R}(A)}; \, {\|x\|}_{A_0} = 1\}$ is
homeomorphic to the set $\{x\in \overline{\mathcal{R}(A)}; \, \|x\| = 1\}$,
which is compact since $\overline{\mathcal{R}(A)}$ is finite-dimensional.
Thus we get the set $\{x\in \overline{\mathcal{R}(A)}; \, {\|x\|}_{A_0} = 1\}$ is compact.

Now, suppose that (ii) holds. Put
$\widetilde{T} = PT\mid_{\overline{\mathcal{R}(A)}}$ and $\widetilde{S} = PS\mid_{\overline{\mathcal{R}(A)}}$.
Therefore, by the equivalence (i)$\Leftrightarrow$(iii) in Theorem \ref{T.2.2},
there exists a sequence of $A_0$-unit vectors $\{x_n\}$ in $\overline{\mathcal{R}(A)}$ such that
\begin{align*}
\displaystyle{\lim_{n\rightarrow +\infty}}{\|\widetilde{T}x_n\|}_{A_0} = {\|\widetilde{T}\|}_{A_0}
\quad \mbox{and} \quad \displaystyle{\lim_{n\rightarrow +\infty}}{\langle \widetilde{T}x_n, \widetilde{S}x_n\rangle}_{A_0} = 0.
\end{align*}
Since the set $\{x\in \overline{\mathcal{R}(A)}; \, {\|x\|}_{A_0} = 1\}$ is compact, hence
$\{x_n\}$ has a subsequence $\{x_{n_k}\}$ that converges to some $x\in \overline{\mathcal{R}(A)}$ with
${\|x\|}_{A_0} = 1$.
This yields
${\|\widetilde{T}x\|}_{A_0} = \displaystyle{\lim_{k\rightarrow +\infty}}{\|\widetilde{T}x_{n_k}\|}_{A_0}
= {\|\widetilde{T}\|}_{A_0}$
and ${\langle \widetilde{T}x, \widetilde{S}x\rangle}_{A_0}
= \displaystyle{\lim_{k\rightarrow +\infty}}{\langle \widetilde{T}x_{n_k}, \widetilde{S}x_{n_k}\rangle}_{A_0} = 0$.
From this it follows that $x\in \mathbb{M}^T_A$ and $Tx\perp_A Sx$.
\end{proof}
%%%%%%%%%%%%%%%%%%%%%%%%
As an immediate consequence of Theorem \ref{T.2.3}, we have the following result.
\begin{corollary}\label{C.2.4}
Let $\mathcal{H}$ be finite dimensional and let $T, S\in\mathbb{B}_{A^{1/2}}(\mathcal{H})$.
Then the following statements are equivalent:
\begin{itemize}
\item[(i)] $T\perp^B_A S$.
\item[(ii)] There exists $x\in \mathbb{M}^T_A$
such that for every $\gamma \in \mathbb{C}$
\begin{align*}
{\|Tx + \gamma Sx\|}^2_A = {\|Tx\|}^2_A + |\gamma|^2{\|Sx\|}^2_A.
\end{align*}
\end{itemize}
\end{corollary}
%%%%%%%%%%%%%%%%%%%%%%%%%%%%%%%%%%
%%%%%%%%%%%%%%%%%%%%%%%%%%%%%%%%%%
\section{Some $A$-distance formulas}
In this section we give some formulas for the $A$-distance of an operator to
the class of multiple scalars of another one in semi-Hilbertian spaces.
For $T, S\in\mathbb{B}_{A^{1/2}}(\mathcal{H})$ we have, by definition,
$d_A(T, \mathbb{C}S) := \displaystyle{\inf_{\gamma \in \mathbb{C}}}{\|T + \gamma S\|}_{A}$.
The following auxiliary lemma is needed for next results.
\begin{lemma}\label{L.3.1}
Let $T, S\in\mathbb{B}_{A^{1/2}}(\mathcal{H})$. Then there exists
$\zeta_0 \in \mathbb{C}$ such that
\begin{align*}
d_A(T, \mathbb{C}S) = {\|T + \zeta_0 S\|}_{A}.
\end{align*}
\end{lemma}
\begin{proof}
If ${\|S\|}_{A} = 0$, then
\begin{align*}
{\|T + \gamma S\|}_{A} \geq {\|T\|}_{A} - |\gamma|{\|S\|}_{A} = {\|T\|}_{A},
\end{align*}
for all $\gamma \in \mathbb{C}$. It is therefore enough to put $\zeta_0 = 0$.
If ${\|S\|}_{A} \neq 0$, then put
$\mathbb{D} := \left\{\gamma \in \mathbb{C}; \, |\gamma| \leq \frac{2{\|T\|}_{A}}{{\|S\|}_{A}}\right\}$
and define $f:\, \mathbb{D}\rightarrow  \mathbb{R}$
by the formula $f(\gamma) = {\|T + \gamma S\|}_{A}$.
Clearly, $f$ is continuous and attains its minimum at, say, $\zeta_0 \in \mathbb{D}$ (of course, there may be
many such points). Then ${\|T + \gamma S\|}_{A} \geq {\|T + \zeta_0 S\|}_{A}$
for all $\gamma \in \mathbb{D}$. If $\gamma \notin \mathbb{D}$,
then $|\gamma| > \frac{2{\|T\|}_{A}}{{\|S\|}_{A}}$. Since $0 \in \mathbb{D}$, we obtain
\begin{align*}
{\|T + \gamma S\|}_{A} \geq |\gamma|{\|S\|}_{A} - {\|T\|}_{A}
> 2{\|T\|}_{A} - {\|T\|}_{A} = {\|T\|}_{A} \geq {\|T + \zeta_0 S\|}_{A}.
\end{align*}
Thus ${\|T + \gamma S\|}_{A} \geq {\|T + \zeta_0 S\|}_{A}$
for all $\gamma \notin \mathbb{D}$. Therefore, ${\|T + \gamma S\|}_{A} \geq {\|T + \zeta_0 S\|}_{A}$
for all $\gamma \in \mathbb{C}$. So, we conclude that
$\displaystyle{\inf_{\gamma \in \mathbb{C}}}{\|T + \gamma S\|}_{A} = {\|T + \zeta_0 S\|}_{A}$ and
hence $d_A(T, \mathbb{C}S) = {\|T + \zeta_0 S\|}_{A}$.
\end{proof}
%%%%%%%%%%%%%%%%%%%%%%%%%%%%
The following result is a kind of the Pythagorean relation for bounded operators in semi-Hilbertian spaces.
\begin{theorem}\label{T.3.2}
Let $T, S\in\mathbb{B}_{A^{1/2}}(\mathcal{H})$ with $m_A(S) > 0$.
Then there exists a unique $\zeta_0\in \mathbb{C}$, such that
\begin{align*}
{\big\|(T + \zeta_0 S) + \gamma S \big\|}^2_A \geq {\|T + \zeta_0 S\|}^2_A + |\gamma|^2\,m^2_A(S)\end{align*}
for every $\gamma \in \mathbb{C}$.
\end{theorem}
\begin{proof}
By Lemma \ref{L.3.1}, there exists
$\zeta_0 \in \mathbb{C}$ such that
\begin{align*}
\displaystyle{\inf_{\gamma \in \mathbb{C}}}{\|T + \gamma S\|}_{A} = {\|T + \zeta_0 S\|}_{A},
\end{align*}
or equivalently,
\begin{align*}
\displaystyle{\inf_{\xi \in \mathbb{C}}}{\|(T + \zeta_0 S) + \xi S\|}_{A} = {\|T + \zeta_0 S\|}_{A}.
\end{align*}
Thus $(T + \zeta_0 S)\perp^B_A S$. So, by the equivalence (i)$\Leftrightarrow$(ii) in Theorem \ref{T.2.2},
for every $\gamma \in \mathbb{C}$, we have
\begin{align*}
{\big\|(T + \zeta_0 S) + \gamma S \big\|}^2_A \geq {\|T + \zeta_0 S\|}^2_A + |\gamma|^2\,m^2_A(S).
\end{align*}
Now, suppose that $\zeta_1$ is another point satisfying the inequality
\begin{align*}
{\big\|(T + \zeta_1 S) + \gamma S \big\|}^2_A \geq {\|T + \zeta_1 S\|}^2_A + |\gamma|^2\,m^2_A(S)
\qquad (\gamma \in \mathbb{C}).
\end{align*}
Choose $\gamma = \zeta_0 - \zeta_1$ to get
\begin{align*}
{\|T + \zeta_0 S\|}^2_A &= {\big\|(T + \zeta_1 S) + (\zeta_0 - \zeta_1) S \big\|}^2_A
\\& \geq {\|T + \zeta_1 S\|}^2_A + |\zeta_0 - \zeta_1|^2\,m^2_A(S)
\\& \geq {\|T + \zeta_0 S\|}^2_A + |\zeta_0 - \zeta_1|^2\,m^2_A(S).
\end{align*}
Hence $0 \geq |\zeta_0 - \zeta_1|^2\,m^2_A(S)$. Since $m^2_A(S)> 0$,
we get $|\zeta_0 - \zeta_1|^2 = 0$, or equivalently, $\zeta_0 = \zeta_1$.
This shows that $\zeta_0$ is unique.
\end{proof}
%%%%%%%%%%%%%%%%%%%%%%%%%%%%%%
Here, we establish one of our main results. In
fact, in what follows, we provide a version of the Bhatia--\v{S}emrl theorem (see \cite[p. 84]{B.S})
in the setting of operators in semi-Hilbertian spaces.
\begin{theorem}\label{T.3.3}
Let $T, S\in\mathbb{B}_{A^{1/2}}(\mathcal{H})$. Then
\begin{align*}
d_A(T, \mathbb{C}S) =
\sup\Big\{|{\langle Tx, y\rangle}_A|; \,  {\|x\|}_{A} = {\|y\|}_{A} = 1, \, Sx \perp_A y\Big\}.
\end{align*}
\end{theorem}
\begin{proof}
Let $x, y \in \mathcal{H}$, ${\|x\|}_{A} = {\|y\|}_{A} = 1$ and let $Sx \perp_A y$.
The Cauchy--Schwarz inequality implies
\begin{align*}
|{\langle Tx, y\rangle}_A| = |{\langle (T + \gamma S)x, y\rangle}_A|\leq
{\|(T + \gamma S)x\|}_{A}{\|y\|}_{A} \leq {\|T + \gamma S\|}_{A}
\end{align*}
for all $\gamma \in \mathbb{C}$. Thus
\begin{align*}
\sup\Big\{|{\langle Tx, y\rangle}_A|; \,  {\|x\|}_{A} = {\|y\|}_{A} = 1, \, Sx \perp_A y\Big\}
\leq {\|T + \gamma S\|}_{A}
\end{align*}
for all $\gamma \in \mathbb{C}$ and so
\begin{align*}
\sup\Big\{|{\langle Tx, y\rangle}_A|; \,  {\|x\|}_{A} = {\|y\|}_{A} = 1, \, Sx \perp_A y\Big\}
\leq \displaystyle{\inf_{\gamma \in \mathbb{C}}}{\|T + \gamma S\|}_{A}.
\end{align*}
Hence
\begin{align}\label{i.2.7.1}
\sup\Big\{|{\langle Tx, y\rangle}_A|; \,  {\|x\|}_{A} = {\|y\|}_{A} = 1, \, Sx \perp_A y\Big\}
\leq d_A(T, \mathbb{C}S).
\end{align}
On the other hand, by Lemma \ref{L.3.1}, there exists
$\zeta_0 \in \mathbb{C}$ such that
$d_A(T, \mathbb{C}S) = {\|T + \zeta_0 S\|}_{A}$.
We assume that $\zeta_0 = 0$ (otherwise we just replace $T$ by $T + \zeta_0 S$).
Thus $d_A(T, \mathbb{C}S) = {\|T\|}_{A}$,
or equivalently, $T \perp^B_A S$.
Then, by the equivalence (i)$\Leftrightarrow$(iii) in Theorem \ref{T.2.2},
there exists a sequence of $A$-unit vectors $\{x_n\}$ in $\mathcal{H}$ such that
$\displaystyle{\lim_{n\rightarrow +\infty}}{\|Tx_n\|}_A = {\|T\|}_A$
and $\displaystyle{\lim_{n\rightarrow +\infty}}{\langle Tx_n, Sx_n\rangle}_A = 0$.
Now, let $Tx_n = \alpha_n Sx_n + \beta_n y_n$
with $Sx_n \perp_A y_n$, ${\|y_n\|}_{A} = 1$, and $\alpha_n, \beta_n\in \mathbb{C}$.
Then we have
\begin{align*}
d^2_A(T, \mathbb{C}S)&= {\|T\|}^2_A
= \displaystyle{\lim_{n\rightarrow +\infty}}{\|Tx_n\|}^2_A
\\& = \displaystyle{\lim_{n\rightarrow +\infty}}{\Big\langle \alpha_n Sx_n + \beta_n y_n, \alpha_n Sx_n + \beta_n y_n\Big\rangle}_A
\\& = \displaystyle{\lim_{n\rightarrow +\infty}}{\langle \alpha_n Sx_n , \alpha_n Sx_n\rangle}_A + |\beta_n|^2
\\& = \displaystyle{\lim_{n\rightarrow +\infty}}{\langle Tx_n - \beta_n y_n, \alpha_n Sx_n\rangle}_A + |\beta_n|^2
\\& = \displaystyle{\lim_{n\rightarrow +\infty}}\alpha_n {\langle Tx_n, Sx_n\rangle}_A -
\overline{\alpha_n}\beta_n {\langle y_n, Sx_n\rangle}_A + |\beta_n|^2
= \displaystyle{\lim_{n\rightarrow +\infty}}|\beta_n|^2.
\end{align*}
Consequently, we obtain
\begin{align*}
d_A(T, \mathbb{C}S) &
= \displaystyle{\lim_{n\rightarrow +\infty}}|\beta_n|
= \displaystyle{\lim_{n\rightarrow +\infty}}|{\langle \beta_n y_n, y_n\rangle}_A|
\\& = \displaystyle{\lim_{n\rightarrow +\infty}}|{\langle Tx_n - \alpha_n Sx_n, y_n\rangle}_A|
= \displaystyle{\lim_{n\rightarrow +\infty}}|{\langle Tx_n, y_n\rangle}_A|
\\& \leq \sup\Big\{|{\langle Tx, y\rangle}_A|; \,  {\|x\|}_{A} = {\|y\|}_{A} = 1, \, Sx \perp_A y\Big\},
\end{align*}
whence
\begin{align}\label{i.2.7.2}
d_A(T, \mathbb{C}S)
\leq \sup\Big\{|{\langle Tx, y\rangle}_A|; \,  {\|x\|}_{A} = {\|y\|}_{A} = 1, \, Sx \perp_A y\Big\}.
\end{align}
From (\ref{i.2.7.1}) and (\ref{i.2.7.2}), we conclude that
\begin{align*}
d_A(T, \mathbb{C}S) =
\sup\Big\{|{\langle Tx, y\rangle}_A|; \,  {\|x\|}_{A} = {\|y\|}_{A} = 1, \, Sx\perp_A y\Big\}.
\end{align*}
\end{proof}
%%%%%%%%%%%%%%%%%%%%%%%%%%%%%%
For $T\in\mathbb{B}(\mathcal{H})$, Fujii and Nakamoto in \cite{Fu.Na} proved that $d_A(T, \mathbb{C}I)$
can be written in the following form:
\begin{align}\label{I.2.7.10}
d(T, \mathbb{C}I) = \Big(\displaystyle{\sup_{\|x\| = 1}}\big(\|Tx\|^2 - |\langle Tx, x\rangle|^2\big)\Big)^{1/2}
= \displaystyle{\sup_{\|x\| = 1}}\big\|Tx - \langle Tx, x\rangle x\big\|,
\end{align}
which shows that $d_A(T, \mathbb{C}I)$ is the supremum over the lengths of all perpendiculars from $Tx$ to $x$,
where $x$ passes over the set of unit vectors.
In the following theorem, for $T, S\in\mathbb{B}_{A^{1/2}}(\mathcal{H})$, we show that
$d_A(T, \mathbb{C}S)$ can also be expressed in the  form generalizing of (\ref{I.2.7.10}).
\begin{theorem}\label{T.2.8}
Let $T, S\in\mathbb{B}_{A^{1/2}}(\mathcal{H})$. Then
\begin{align*}
d^2_A(T, \mathbb{C}S)=
\displaystyle{\sup_{{\|x\|}_A = 1}}\Phi^{(T, S)}_A(x),
\end{align*}
where
\begin{align*}
\Phi^{(T, S)}_A(x)=\begin{cases}
{\|Tx\|}^2_A - \frac{|{\langle Tx, Sx\rangle}_A|^2}{{\|Sx\|}^2_A} &\text{if\, ${\|Sx\|}_A\neq0$}\\
{\|Tx\|}^2_A       &\text{if\, ${\|Sx\|}_A=0$}.
\end{cases}
\end{align*}
\end{theorem}
\begin{proof}
For every $\gamma \in \mathbb{C}$ and every $A$-unit vector $x\in \mathcal{H}$
such that ${\|Sx\|}_A\neq0$, we have
\begin{align*}
{\|Tx + \gamma Sx\|}^2_A &- \frac{|{\langle Tx + \gamma Sx, Sx\rangle}_A|^2}{{\|Sx\|}_A^2}
\\&= {\|Tx\|}^2_A + |\gamma|^2{\|Sx\|}^2_A + 2\mbox{Re}{\langle Tx, \gamma Sx\rangle}_A
\\& \qquad \qquad - \frac{|{\langle Tx, Sx\rangle}_A|^2 + |\gamma|^2{\|Sx\|}^4_A
+ 2{\|Sx\|}^2_A\mbox{Re}{\langle Tx, \gamma Sx\rangle}_A}{{\|Sx\|}^2_A}
\\& = {\|Tx\|}^2_A - \frac{|{\langle Tx, Sx\rangle}_A|^2}{{\|Sx\|}^2_A}.
\end{align*}
Thus
\begin{align*}
\Phi^{(T, S)}_A(x) &= {\|Tx + \gamma Sx\|}^2_A - \frac{|{\langle Tx + \gamma Sx, Sx\rangle}_A|^2}{{\|Sx\|}^2_A}
\\&\leq {\|Tx + \gamma Sx\|}^2_A \leq {\|T + \gamma S\|}^2_A.
\end{align*}
Also, in the case ${\|Sx\|}_A=0$ we have
\begin{align*}
\Phi^{(T, S)}_A(x) = {\|Tx\|}^2_A \leq \big({\|Tx + \gamma Sx\|}_A + {\|\gamma Sx\|}_A\big)^2
= {\|Tx + \gamma Sx\|}^2_A\leq {\|T + \gamma S\|}^2_A.
\end{align*}
Hence we obtain $\Phi^{(T, S)}_A(x) \leq {\|T + \gamma S\|}^2_A$ for every $A$-unit vector $x\in \mathcal{H}$
and every $\gamma \in \mathbb{C}$. Therefore,
$\displaystyle{\sup_{{\|x\|}_A = 1}}\Phi^{(T, S)}_A(x) \leq {\|T + \gamma S\|}^2_A$ for every $\gamma \in \mathbb{C}$
and consequently,
\begin{align*}
\displaystyle{\sup_{{\|x\|}_A = 1}}\Phi^{(T, S)}_A(x) \leq \displaystyle{\inf_{\gamma \in \mathbb{C}}}{\|T + \gamma S\|}^2_A.
\end{align*}
Thus
\begin{align}\label{i.28.1}
\displaystyle{\sup_{{\|x\|}_A = 1}}\Phi^{(T, S)}_A(x) \leq d^2_A(T, \mathbb{C}S).
\end{align}
Now, take $A$-unit vectors $x, y \in \mathcal{H}$ such that $Sx \perp_A y$.
If ${\|Sx\|}_A=0$, then
\begin{align*}
|{\langle Tx, y\rangle}_A|^2 \leq {\|Tx\|}^2_A{\|y\|}^2_A = \Phi^{(T, S)}_A(x) \leq \displaystyle{\sup_{{\|x\|}_A = 1}}\Phi^{(T, S)}_A(x).
\end{align*}
If ${\|Sx\|}_A\neq0$, then
\begin{align*}
|{\langle Tx, y\rangle}_A|^2 &
= \left|{\Big\langle Tx - \frac{{\langle Tx, Sx\rangle}_A}{{\|Sx\|}^2_A}Sx, y\Big\rangle}_A\right|^2
\\& \leq {\Big\langle Tx - \frac{{\langle Tx, Sx\rangle}_A}{{\|Sx\|}^2_A}Sx, Tx - \frac{{\langle Tx, Sx\rangle}_A}{{\|Sx\|}^2_A}Sx\Big\rangle}_A
\\& = {\|Tx\|}^2_A - \frac{|{\langle Tx, Sx\rangle}_A|^2}{{\|Sx\|}^2_A}
= \Phi^{(T, S)}_A(x) \leq \displaystyle{\sup_{{\|x\|}_A = 1}}\Phi^{(T, S)}_A(x).
\end{align*}
So, we conclude that $|{\langle Tx, y\rangle}_A|^2 \leq \displaystyle{\sup_{{\|x\|}_A = 1}}\Phi^{(T, S)}_A(x)$
for all $A$-unit vectors $x, y \in \mathcal{H}$ such that $Sx \perp_A y$.
Therefore,  Theorem \ref{T.3.3} implies that
\begin{align}\label{i.28.2}
d^2_A(T, \mathbb{C}S) \leq \displaystyle{\sup_{{\|x\|}_A = 1}}\Phi^{(T, S)}_A(x).
\end{align}
Now, the result follows from (\ref{i.28.1}) and (\ref{i.28.2}).
\end{proof}
%%%%%%%%%%%%%%%%%%%%%%%%%%%%%%
We close this paper with the following Inf-sup equality in semi-Hilbertian spaces.
%%%%%%%%%%%%%%%%%%%%%%%%%%%%%%
\begin{theorem}
Let $T, S\in\mathbb{B}_{A^{1/2}}(\mathcal{H})$. Then
\begin{align*}
\displaystyle{\inf_{\gamma \in \mathbb{C}}}\,\displaystyle{\sup_{{\|x\|}_A = 1}}{\|(T + \gamma S)x\|}^2_A =
\displaystyle{\sup_{{\|x\|}_A = 1}}\,\displaystyle{\inf_{\gamma \in \mathbb{C}}}{\|(T + \gamma S)x\|}^2_A.
\end{align*}
\end{theorem}
\begin{proof}
Let $x\in \mathcal{H}$ with ${\|x\|}_A = 1$. If ${\|Sx\|}_A = 0$, then
\begin{align*}
{\|(T + \gamma S)x\|}_A \geq {\|Tx\|}_A  - |\gamma|{\|Sx\|}_A = {\|Tx\|}_A,
\end{align*}
for all $\gamma \in \mathbb{C}$. Thus
\begin{align*}
{\|Tx\|}^2_A \geq \displaystyle{\inf_{\gamma \in \mathbb{C}}}{\|(T + \gamma S)x\|}^2_A \geq {\|Tx\|}^2_A,
\end{align*}
whence $\displaystyle{\inf_{\gamma \in \mathbb{C}}}{\|(T + \gamma S)x\|}^2_A = {\|Tx\|}^2_A$.
Hence $\displaystyle{\inf_{\gamma \in \mathbb{C}}}{\|(T + \gamma S)x\|}^2_A = \Phi^{(T, S)}_A(x)$.

If ${\|Sx\|}_A \neq 0$, then simple computations show that
\begin{align*}
{\|(T + \gamma S)x\|}^2_A = {\|Sx\|}^2_A\Big|\frac{{\langle Tx, Sx\rangle}_A}{{\|Sx\|}^2_A} + \gamma\Big|^2
+ {\|Tx\|}^2_A - \frac{|{\langle Tx, Sx\rangle}_A|^2}{{\|Sx\|}^2_A}.
\end{align*}
Thus ${\|(T + \gamma S)x\|}^2_A$ achieves its minimum at $-\frac{{\langle Tx, Sx\rangle}_A}{{\|Sx\|}^2_A}$
and the minimum value is
${\|Tx\|}^2_A - \frac{|{\langle Tx, Sx\rangle}_A|^2}{{\|Sx\|}^2_A}$.
Hence $\displaystyle{\inf_{\gamma \in \mathbb{C}}}{\|(T + \gamma S)x\|}^2_A = \Phi^{(T, S)}_A(x)$
for every $A$-unit vector $x\in \mathcal{H}$.
From this, by Theorem \ref{T.2.8}, we conclude that
\begin{align*}
\displaystyle{\sup_{{\|x\|}_A = 1}}\,\displaystyle{\inf_{\gamma \in \mathbb{C}}}{\|(T + \gamma S)x\|}^2_A
&= \displaystyle{\sup_{{\|x\|}_A = 1}}\Phi^{(T, S)}_A(x)
\\&= d^2_A(T, \mathbb{C}S)
\\&= \displaystyle{\inf_{\gamma \in \mathbb{C}}}{\|T + \gamma S\|}^2_A
= \displaystyle{\inf_{\gamma \in \mathbb{C}}} \,\displaystyle{\sup_{{\|x\|}_A = 1}}{\|(T + \gamma S)x\|}^2_A.
\end{align*}
\end{proof}
%%%%%%%%%%%%%%%%%%%%%%%
{\bf Acknowledgments.} The author would like to thank the referees for their valuable
comments which helped to improve the paper.
This work was supported by a grant from Shanghai Municipal Science and Technology Commission (18590745200).

\bibliographystyle{amsplain}

\end{document}